\documentclass[11pt]{article}
\usepackage{a4wide}
\usepackage{amsmath,amssymb,amsfonts,amsthm}
\usepackage[usenames]{color}
\usepackage{graphicx}

\usepackage[colorlinks=true,citecolor=black,linkcolor=black,urlcolor=blue]{hyperref}

\usepackage{float}

\usepackage{enumerate} 
\usepackage{etoolbox}
\makeatletter
\AfterEndEnvironment{thm}{\everypar{\setbox\z@\lastbox\everypar{}}}
\AfterEndEnvironment{theorem}{\everypar{\setbox\z@\lastbox\everypar{}}}
\AfterEndEnvironment{lemma}{\everypar{\setbox\z@\lastbox\everypar{}}}
\AfterEndEnvironment{definition}{\everypar{\setbox\z@\lastbox\everypar{}}}
\AfterEndEnvironment{remark}{\everypar{\setbox\z@\lastbox\everypar{}}}
\AfterEndEnvironment{proposition}{\everypar{\setbox\z@\lastbox\everypar{}}}
\AfterEndEnvironment{corollary}{\everypar{\setbox\z@\lastbox\everypar{}}}
\AfterEndEnvironment{fact}{\everypar{\setbox\z@\lastbox\everypar{}}}
\makeatother

\newenvironment{remark}{\noindent {\bf Remark}.}{\par\smallskip\par}

\newcommand{\Gnh}{{G_{n, 1/2}}}

\newtheorem{firsttheorem}{Proposition}

\newtheorem{theorem}[firsttheorem]{Theorem}
\newtheorem{lemma}[firsttheorem]{Lemma}
\newtheorem{corollary}[firsttheorem]{Corollary}
\newtheorem{conjecture}[firsttheorem]{Conjecture}

\newtheorem{proposition}[firsttheorem]{Proposition}

\newtheorem{definition}[firsttheorem]{Definition}

\newtheorem*{lemma*}{Lemma}

\usepackage{marginnote}

\usepackage{bbm}
\usepackage[numbers]{natbib}

\newcommand{\kp}{\mathbf{k}}
\newcommand{\bfk}{\mathbf{k}}

\newcommand{\Pb}{{\mathbb{P}}}

\newcommand{\E}{\mathbb{E}}

\newcommand{\boldk}{{\boldsymbol{k}}}

\renewcommand{\le}{\leqslant}
\renewcommand{\ge}{\geqslant}

\renewcommand{\epsilon}{\varepsilon}

\usepackage{xcolor}

\makeatletter
\def\moverlay{\mathpalette\mov@rlay}
\def\mov@rlay#1#2{\leavevmode\vtop{%
		\baselineskip\z@skip \lineskiplimit-\maxdimen
		\ialign{\hfil$\m@th#1##$\hfil\cr#2\crcr}}}
\newcommand{\charfusion}[3][\mathord]{
	#1{\ifx#1\mathop\vphantom{#2}\fi
		\mathpalette\mov@rlay{#2\cr#3}
	}
	\ifx#1\mathop\expandafter\displaylimits\fi}
\makeatother

\newcommand{\bigcupdot}{\charfusion[\mathop]{\bigcup}{\cdot}}

\begin{document}

\title{The difference between the chromatic and the cochromatic number of a random graph}

\author{
Annika Heckel\thanks{Matematiska institutionen, Uppsala universitet, Box 480, 751 06 Uppsala, Sweden. Email: \texttt{annika.heckel@math.uu.se}. Funded by the Swedish Research Council, Starting Grant 2022-02829.}
}

\maketitle
	
\begin{abstract}
The cochromatic number $\zeta(G)$ of a graph $G$ is the minimum number of colours needed for a vertex colouring where every colour class is either an independent set or a clique. Let $\chi(G)$ denote the usual chromatic number. Around 1991 Erd\H{o}s and Gimbel asked: For the random graph $G \sim G_{n, 1/2}$, does $\chi(G)-\zeta(G) \rightarrow \infty$ whp? Erd\H{o}s offered \$100 for a positive and \$1,000 for a negative answer.
 
 We give a positive answer to this question for roughly 95\% of all values $n$. 
\end{abstract}

\section{Introduction}
Given a graph $G$, the cochromatic number $\zeta(G)$  is the minimum number of colours needed for a vertex colouring where every colour class is either an independent set or a clique. If $\chi(G)$ denotes the usual chromatic number, then clearly $\zeta(G) \le \chi(G)$. By Ramsey's theorem, any graph on $n$ vertices contains either a clique or an independent set of order $\log n$, and so (by colouring greedily) the cochromatic number of any $n$-vertex graph is at most of order~$n/ \log n$.

In the following, we consider random graphs and let  $G \sim G_{n, 1/2}$.  Since the maximum size of a clique and of an independent set in $G$ are both with high probability\footnote{As usual, we say that a sequence $(E_n)_{n \ge 0}$ of events holds \emph{with high probability (whp)} if $\Pb(E_n) \rightarrow 1$ as $n \rightarrow \infty$.} asymptotically equal to $2 \log_2 n$ (see \S\ref{section:alpha} below), and as $\chi(G) \sim \frac n {2\log_2n}$ by a classic result of Bollob\'as \cite{bollobas1988chromatic}, it follows that, whp,
\[
(1+o(1)) \frac{n}{2\log_2 n} \le \zeta(G) \le \chi(G) =(1+o(1)) \frac{n}{2\log_2 n}.
\]
So $\zeta(\Gnh) \sim \chi(\Gnh) \sim \frac n{2\log_2 n}$, that is, the ratio $\chi(G_{n, 1/2})/\zeta(G_{n,1/2})$ tends to $1$ whp. 

 Erd\H{o}s and Gimbel \cite{erdos1993some} asked the following question: does the \emph{difference} $\chi(G)-\zeta(G)$ tend to infinity as $n \rightarrow \infty$? In other words, is there a function $f(n) \rightarrow \infty$ such that, with high probability,
\begin{equation}\label{eq:question}
\chi(G) - \zeta(G) > f(n)?
\end{equation}
At a conference on random graphs in Pozna\'n\footnote{most likely in 1991, based on Erd\H{o}s's and Gimbel's participation records}, Erd\H{o}s offered \$100 for the solution if the answer was `yes', and \$1000 if the answer was `no' (although later said to Gimbel that perhaps \$1000 was too much) --- see Gimbel's recollections in \cite{gimbel2016some}. The question is listed as Problem \#625 on Thomas Bloom's Erd\H{o}s Problems website \cite{erdosproblem625}.

Recently,  Raphael Steiner~\cite{steiner2024cochromatic} and, independently, the author of the present paper \cite{heckel2024question} noted the following connection between $\chi(G)-\zeta(G)$ and the concentration interval length of the chromatic number $\chi(G_{n, 1/2})$:  if 
\begin{equation}\chi(G) - \zeta(G) \le g(n) \text{  whp for some function $g(n)$,} \label{eq:gn}
	\end{equation}
then there is a sequence of intervals of length $g(n)$ which contains $\chi(G)$ whp. We can therefore use the lower bounds on the concentration interval length of $\chi(G_{n, 1/2})$ from \cite{heckel2019nonconcentration,HRHowdoes} to get corresponding statements about $g(n)$: For example, given a function $g(n)$ as in \eqref{eq:gn}, there is an infinite sequence of integers $n^*$ such that $g(n^*) \ge c \sqrt{n^*} \log \log n^* / \log^3 n^*$ for some constant $c>0$. This observation suggests that the answer to Erd\H{o}s and Gimbel's question may be `yes', at least for these values $n^*$ --- but it does not imply it for any $n$.

In this paper, we show that the answer to \eqref{eq:question}  is `yes' for roughly 95\% of all values of $n$. To state the result, we need some notation. Set 
\begin{equation}
	\alpha_0 = \alpha_0(n) = 2 \log_2 n - 2 \log_2 \log_2 n +2 \log_2 (e/2)+1, ~~~~ \alpha= \left \lfloor \alpha_0 \right \rfloor, \label{eq:defalpha0}
\end{equation}
and let
\[
\mu_\alpha = {n \choose \alpha} (\tfrac 12) ^{\alpha \choose 2}
\] 
be the expected number of independent sets (or  by symmetry, the expected number of cliques) of size $\alpha$ in $G_{n, 1/2}$. Then our result is as follows.
\begin{theorem}\label{theorem:difference}
Fix $\epsilon>0$, and let $n$ be such that $n^{0.05+\epsilon}\le \mu_\alpha \le n^{1-\epsilon}$. Let $G \sim \Gnh$, then whp, 
\[
\chi(G) - \zeta(G) \ge n^{1-\epsilon}.
\]	
\end{theorem}
The condition of this theorem holds for roughly 95\% of all values $n$, see \S\ref{section:alpha} below. The proof relies on transferring recent results by the author and Konstantinos Panagiotou \cite{heckel2023colouring} 
on the usual chromatic number to the cochromatic setting. We discuss the remaining $5\%$ of values $n$ in \S\ref{section:discussion} --- we conjecture that Theorem~\ref{theorem:difference} is true for all $n$, and that the lower bound $n^{1-\epsilon}$ can be replaced by $\Theta(n / \log^3 n)$.

The paper is organised as follows. After setting up some notation and background in \S\ref{section:setup}, we explain the proof structure in \S\ref{section:proof}, and prove Theorem~\ref{theorem:difference} conditional on the key Proposition~\ref{proposition:main}. The main part of the paper is the proof of Proposition~\ref{proposition:main} in \S\ref{section:propproof}, where we show how to transfer previous work on the chromatic number to the cochromatic setting. We conclude with a discussion in \S\ref{section:discussion}.

\section{Setup and notation}\label{section:setup}
Throughout this paper, we set $m= \left \lfloor \frac 12 N\right \rfloor$, and we let $\Pb_{1/2}$ and $\Pb_m$ refer to probabilities and $\E_{1/2}$ and $\E_m$ to expectations in $G_{n, 1/2}$ and $G_{n,m}$, respectively.

\subsection{Background on the independence number of $G_{n, 1/2}$}\label{section:alpha}

We briefly review some background on maximal independent sets (and by symmetry, maximal cliques) in $G_{n, 1/2}$. 
As usual, let $\alpha(G)$ denote the independence number of a graph $G$. Bollob\'as and Erd\H{o}s \cite{erdoscliques}, and independently Matula \cite{matula1970complete,matula1972employee}, proved that, whp, $\alpha(G_{n,1/2}) = \lfloor \alpha_0+o(1) \rfloor$, where $\alpha_0$ was defined in \eqref{eq:defalpha0}. 

For $t \ge 1$, let
\[
\mu_t = {n \choose t} 2^{-{t \choose 2}}
\]
denote the expected number of independent sets of size $t$ in $G_{n, 1/2}$. If we interpret $\mu_t$ appropriately for real $t$, then $\alpha_0$ is, to a good approximation, the value of $t$ so that $\mu_t \approx 1$. Furthermore, it is not hard to show that for $t=\alpha_0+O(1)$,
\[
\frac{\log \mu_t}{\log n} = \alpha_0-t+o(1), 
\]
and that, recalling $\alpha=\lfloor \alpha_0\rfloor$,
\begin{equation}\label{eq:mualpha}
1 \le \mu_\alpha \le n^{1+o(1)}, ~~~~ \alpha_0 - \alpha = \frac{\log \mu_\alpha}{\log n}+o(1), ~~~~\text{and }~~~~ \mu_{\alpha-1} = \Theta(n \mu_\alpha /\log n).
\end{equation}
So if $n^{0.05+\epsilon}\le \mu_\alpha \le n^{1-\epsilon}$ as in Theorem~\ref{theorem:difference}, then whp $ \alpha(G_{n, 1/2})=\alpha$ and furthermore $\mu_{\alpha-1} \ge n^{1.05}$ if $n$ is large enough.

It follows from \eqref{eq:defalpha0} and \eqref{eq:mualpha} that as we increase $n$ and thereby $\alpha_0(n)$, $\mu_\alpha$ increases from $n^{o(1)}$ to $n^{1+o(1)}$ (the exponent is roughly linear in $\log n$), and then drops back to $n^{o(1)}$. The drop happens whenever $\alpha_0(n)$ is near an integer and $\alpha(G_{n, 1/2}) = \lfloor \alpha_0+o(1) \rfloor$ `jumps up' to this next integer.

 We claimed earlier that the condition of Theorem~\ref{theorem:difference} is true for roughly 95\% of all $n$. Indeed, as we increase $n$, the fraction of $n' \le n$ that  Theorem~\ref{theorem:difference} applies to (picking some arbitrarily small $\epsilon>0$) fluctuates between
 \[
 \frac{2^{-0.05/2}-2^{-0.5}}{1-2^{-0.5}} \approx 0.9413  \quad \text{and} \quad \frac{1-2^{-0.95/2}}{1-2^{-0.5}} \approx 0.9578.
 \]

\subsection{(Co-)colouring profiles and the $t$-bounded chromatic number}
We view colourings and cocolourings as ordered partitions 
\[
\Pi=(V_1, \dots, V_k)
\]
of the vertex set $[n]$. In the case of colourings, each $V_i$ is an independent set, whereas for a cocolouring each $V_i$ is either an independent set or a clique. We will also consider partial ordered partitions
\[
\Pi=(V_1, \dots, V_\ell)
\]
where all $V_i$ are independent sets (or either an independent set or a clique), and $\bigcup_i V_i \subset V$.

We define \emph{profiles} to describe the sizes of the sets $V_i$. Given $k \ge 1$, a \emph{$k$-profile} is a sequence $\bfk=(k_u)_{1 \le  u \le n}$ such that 
\[
\sum_{1\le u \le n} k_u = k \quad \text{and} \quad \sum_{1\le u \le n} uk_u \le n.
\]
If $\sum_{1\le u \le n} uk_u = n$ we call $\bfk$ a complete profile, otherwise a partial profile.

We say that a (partial or complete) partition $\Pi=(V_1, \dots, V_k)$ has profile $\bfk$ if (1) there are exactly $k_u$ sets $V_i$ so that $|V_i|=u$, and (2) the sets $V_i$ are decreasing in size. A profile $\bfk$ is called \emph{$t$-bounded} if 
\[k_u=0 \quad \text{for} \quad u>t.\]
We will consider $(\alpha-1)$-bounded profiles for most of the paper.

The \emph{$t$-bounded chromatic number} $\chi_t(G)$ is the least number of colours $k$ so that there is a colouring with $k$ colours where every colour class contains at most $t$ vertices (that is, a colouring with a $t$-bounded $k$-profile).

\subsection{Counts of colourings and co-colourings}

Give a  profile $\bfk$, let $X_{\bfk}$ count the number of colourings and $X_{\bfk}^\mathrm{co}$ the number of cocolourings with profile $\bfk$ in either $G_{n, 1/2}$ or $G_{n,m}$, and let
\begin{equation}\label{eq:unordered}
\bar X_\bfk = \frac{X_\bfk}{\prod_{u=1}^t k_u!}, \quad \quad \quad \bar X^{\mathrm{co}}_\bfk = \frac{X_\bfk^\mathrm{co}}{\prod_{u=1}^t k_u!} 
\end{equation}
count the number of unordered (co-)colourings with profile $\bfk$. The expected number of colourings with profile $\bfk$ in $G_{n, 1/2}$ and in $G_{n,m}$ is given by 
\begin{equation}
\mathbb{E}_{1/2}[X_\bfk] = P_{\bfk }2^{-f_\bfk}
\quad
\text{and}
\quad
\mathbb{E}_m[X_\bfk] = P_\bfk \, \frac{{N-f_\bfk \choose m}}{{N \choose m}} \label{eq:expectation}
\end{equation}
respectively, where
\begin{equation*}
P_\bfk= \frac{n!}{(n- \sum_{1\le u \le n}k_u u)!\prod_{1 \le u \le n} u!^{k_u}}   \quad \text{and} \quad	f_\mathbf{k}= \sum_{1 \le u \le n} {u \choose 2} k_u.
\end{equation*}
Of course the unordered expressions are then given by 
\[
\mathbb{E}_{1/2}[\bar X_\bfk] =  \frac{\mathbb{E}_{1/2}[ X_\bfk]}{\prod_{1 \le u \le n}k_u!} 
\quad
\text{and}
\quad
\mathbb{E}_{m}[\bar X_\bfk] =  \frac{\mathbb{E}_{m}[ X_\bfk]}{\prod_{1 \le u \le n}k_u!} .
\]

Let $E_{n,k,t}$ denote the expected total number of unordered $t$-bounded $k$-colourings of $G_{n, 1/2}$ (from all $t$-bounded $k$-profiles). 
Then the \emph{$t$-bounded first moment threshold} is defined as
\begin{equation}\label{ktdef}
	\boldk_t(n) := \min\{k: E_{n,k,t} \ge 1 \}.
\end{equation}

\subsection{Technical lemma}

The following lemma is useful to transfer probabilities between $G_{n, 1/2}$ and $G_{n,m}$.
\begin{lemma}[\cite{heckel2023colouring}, Le.~3.7] \label{gnmlemma}
	Let $p \in (0,1)$ be constant and $m = Np + O(1)$. If $x=x(n)=o(n^{4/3})$, then 
	\[
	\frac{{N-x \choose m}}{{N \choose m}} \sim q^{x} \exp \left(- \frac{(b-1)x^2}{n^2}\right)
	\]
	where $q=1-p$ and $b=1/q$.
\end{lemma}

\section{Proof of Theorem~\ref{theorem:difference}}\label{section:proof}

\subsection{Proof strategy}

The proof relies on finding two values $k_1=k_1(n)$ and $k_2=k_2(n)$ so that, if $n^{0.05+\epsilon} \le \mu_\alpha \le n^{1-\epsilon}$,
\begin{enumerate}
	\item[(A)] $k_1 - k_2 \ge n^{1-\epsilon}$ if $n$ is large enough, 
	\item[(B)] whp, $\chi(G_{n, 1/2}) \ge k_1$, and	
	\item[(C)] whp, $\zeta(\Gnh) \le k_2$.
\end{enumerate}

\subsubsection{The lower bound on $\chi(G_{n, 1/2})$}

%
%

Our starting point is the following result from \cite{heckel2023colouring}.

\begin{lemma}[\cite{heckel2023colouring}, Lemma~8.1]\label{lemma:lowerbound}
	Let $G \sim G_{n, 1/2}$, then, whp,
	\[
	\chi_{\alpha-1} (G)\ge \boldsymbol{k}_{\alpha-1}-1.
	\]
\end{lemma}
It is an easy observation that for any graph $G$ with independence number $\alpha(G)$ and $X_{\alpha(G)}$ independent sets of size $\alpha(G)$, 
\[\chi_{\alpha(G)-1}(G) \le \chi(G)+X_{\alpha(G)}.\]
This is because any colouring of $G$ contains at most $X_{\alpha(G)}$ colour classes of size $\alpha(G)$ (and no larger ones), and so we can obtain an $(\alpha(G)-1)$-bounded colouring by removing one vertex from each such colour class and giving it a new colour, introducing at most $X_{\alpha(G)}$  new colours in total. 
 So we have the following consequence of Lemma~\ref{lemma:lowerbound} (bounding $X_{\alpha} \le n^{1-0.99\epsilon}$ whp by the first moment method).
\begin{corollary}
Fix $\epsilon \in (0,1)$, let $n$ be such that $n^{0.05+\epsilon}\le \mu_\alpha \le n^{1-\epsilon}$, and let $G \sim G_{n, 1/2}$. Then, whp,
	\[
	\chi (G) \ge \boldk_{\alpha-1}-n^{1-0.9\epsilon}.
	\]		
\end{corollary}
So, under the conditions of Theorem~\ref{theorem:difference}, we can pick
\begin{equation} \label{eq:k1}
k_1= k_1(n) =\boldk_{\alpha-1}-n^{1-0.9\epsilon}
\end{equation}
and by the corollary, (B) holds.

\begin{remark}
	We could also use a slightly stronger lower bound for $\chi(G_{n,1/2})$ from \cite{HRHowdoes}, Lemmas~43 and 41, and then set $k_1=\boldk_{\alpha-1} - 2 \mu_\alpha/\alpha$, say, but this is not necessary for the proof of Theorem~\ref{theorem:difference}.
\end{remark}

\subsubsection{The upper bound on $\zeta(\Gnh)$}

We prove this bound by the second moment method. Recall the Paley-Zygmund inequality,
\begin{equation}\label{eq:paleyzygmund}
\Pb(Z>0) \ge \E[Z]^2/\E[Z^2],
\end{equation}
which holds for any random variable $Z \ge 0$ with finite variance. We will show the following proposition.

\begin{proposition}\label{proposition:main}
	Let $\epsilon>0$, suppose $n^{0.05+\epsilon} \le \mu_\alpha \le n^{1-\epsilon}$, and set
	\begin{equation}\label{eq:defkst}
		k^*= \boldk_{\alpha-1} - n^{1-\epsilon/2}. 
	\end{equation} 
There is an $(\alpha-1)$-bounded $k^*$-profile $\bfk^*=(k_u^*)_{u=1}^{\alpha-1}$ and a random variable $Z \ge 0$ so that 
\begin{enumerate}
	\item[(a)] $Z>0$ implies $X_{\bfk^*}^\mathrm{co} >0$, and
	\item[(b)] $\E_{1/2}[Z^2]/\E_{1/2}[Z]^2 <\exp(n^{0.99})$. 
\end{enumerate}	
\end{proposition}
Proposition~\ref{proposition:main} and \eqref{eq:paleyzygmund} imply that
\begin{equation}\label{eq:probbound}
\Pb(\zeta(G_{n, 1/2}) \le k^*) = \Pb(X^\mathrm{co}_{\bfk^*} >0) > \exp \big( -n^{0.99}\big).
\end{equation}
This is not quite what we want yet, but with a standard trick first used by Frieze in \cite{frieze1990independence}, we can get a lower probability bound $1-o(1)$ by increasing the number of colours a little. For $G \sim G_{n, 1/2}$, consider the \emph{vertex exposure martingale} (for full details, see \S2.4 in \cite{janson:randomgraphs}): We reveal information on $G_{n, 1/2}$ step by step by revealing which edges are present between the first $k$ vertices, and we condition $\zeta(G)$ on this information, which defines a martingale. This martingale has bounded differences: Changing all the edges incident with any particular vertex can only change $\zeta(G)$ by $\pm1$ (since we can `fix' any given colouring by placing this vertex in its own colour class). Therefore, we may apply the Azuma-Hoeffding inequality (see Theorem 2.25 and Corollary 2.27 ~\cite{janson:randomgraphs}), giving for any $t >0$ that
\begin{equation} \label{eq:azumabound}
\Pb(|\zeta(G)-\E[\zeta(G)]|\ge t) \le 2 \exp (-t^2/(2n)).
\end{equation}
Comparing this bound with $t=n^{0.999}$ to \eqref{eq:probbound}, we must have $k^* \ge \E[\zeta(G_{n, 1/2})]-n^{0.999}$. Applying \eqref{eq:azumabound} again with $t=n^{0.999}$ shows that $P(\zeta(G_{n, 1/2})>k^*+2n^{0.999}) \rightarrow 0$, so whp,
\[
\zeta(G_{n, 1/2}) \le k^* + 2n^{0.999} =: k_2,
\]
establishing (C). Using this definition of $k_2$, \eqref{eq:defkst} and \eqref{eq:k1}, 
we have
\[
k_1-k_2 = n^{1-\epsilon/2}-2n^{0.999}-n^{1-0.9 \epsilon}\ge n^{1-\epsilon}
\]
if $\epsilon>0$ is small enough (which we can assume wlog) and $n$ is large enough. So (A) is fulfilled as well, giving Theorem~\ref{theorem:difference}.

It remains to prove Proposition~\ref{proposition:main} in order to complete the proof of Theorem~\ref{theorem:difference}.

\section{Proof of Proposition~\ref{proposition:main}} \label{section:propproof}

For a profile $\bfk$, let $\Pi_\bfk$ denote the set of ordered partitions of the vertex set $[n]$ with profile $\bfk$. Define the events
\begin{align*}
A_\pi &: \quad \text{$\pi$ is a colouring} \\
A^\textrm{co}_{\pi} &: \quad \text{$\pi$ is a cocolouring}.
\end{align*}
We will use the techniques developed in \cite{heckel2023colouring} to bound the second moment, translating the results via the following correspondence.

\begin{proposition}\label{prop:probabilities}
	For $k \ge 1$, let $\kp$ be a $k$-profile with $k_1=0$, and $\pi, \pi' \in \Pi_{\bfk}$. Then
\[\Pb_{1/2}(A_\pi^{\mathrm{co}}) = 2^k 	\Pb_{1/2}(A_\pi)\]
and, if $\pi$ and $\pi'$ have exactly $\ell$ parts in common, where $0 \le \ell \le k$, then
\[	\Pb_{1/2}(A_\pi^{\mathrm{co}} \cap A_{\pi'}^{\mathrm{co}}) \le 2^{2k-\ell} 	\Pb_{1/2}(A_\pi \cap A_{\pi'}).
\]
\end{proposition}
\begin{proof}
	The first part follows easily by observing that once we have chosen which of the $k$ colour classes of $\pi$ are cliques and which are independent sets, the probability that $\pi$ is a cocolouring given these choices is exactly $\Pb_{1/2}(A_{\pi})$. As $k_1=0$, every colour class contains at least one potential edge, so the events associated with a particular choice are disjoint, giving $\Pb_{1/2}(A^\mathrm{co}_{\pi})=2^k \Pb_{1/2}(A_\pi)$. For the second part, observe that there are $2^{2k-\ell}$ possible choices as $\pi$ and $\pi'$ share $\ell$ parts. Not all of these choices for cliques and independent sets in $\pi$ and $\pi'$ may be compatible with each other (a designated `clique' in $\pi$ may share at least two vertices with an `independent set' in $\pi'$, or vice versa), but if they are, then the probability that both $\pi$ and $\pi'$ are cocolourings with these choices is exactly $\Pb(A_{\pi} \cap A_{\pi'})$.
	\end{proof}
From this proposition, it follows that for any $k$-profile $\bfk$ with $k_1=0$,
\begin{equation}\label{eq:firstmomentcocol}
\E_{1/2}[X_\bfk^\mathrm{co}] = \sum_{\pi \in \Pi_\bfk}\Pb(A_\pi^\mathrm{co})= 2^k \E_{1/2}[X_\bfk] .
\end{equation}
For the second moment, we could of course also bound
\[
\E_{1/2}[({X_\bfk^\mathrm{co}})^2] = \sum_{\pi, \pi' \in \Pi_\bfk}\Pb(A^\mathrm{co}_\pi \cap A^\mathrm{co}_\pi)\le 2^{2k} \E_{1/2}[X_\bfk^2],
\]
but this is not enough: We want to apply the second moment method to a $k^*$ which is smaller than the typical value of the chromatic number of $G_{n, 1/2}$, so the second moment moment can not work for normal $k$-colourings. Nevertheless, Proposition~\ref{prop:probabilities} is extremely  useful for reusing the second moment bounds established 
 in \cite{heckel2023colouring}, we simply need to split up the calculation into cases depending on how many parts $\pi$ and $\pi'$ have in common.
 
\subsection{Tame profiles} 

The notion of a tame profile was introduced in \cite{HRHowdoes}: These are bounded profiles which satisfy certain smoothness properties (a tail bound and a lower bound on the expected number of colourings). It turns out that we do not lose much by restricting ourselves to these tame profiles, as the `optimal' profiles $\bfk$ which maximise or near-maximise the expectation $\E[\bar X_\bfk]$ are tame.

The definition is somewhat technical and we state it here for completeness, although for the purposes of reading this paper it suffices to observe that by Lemma~\ref{lemma:kstartame} later, the profile $\bfk^*$ that we will consider (which is close to optimal) is tame. We restrict the definition to $(\alpha-1)$-bounded profiles here, as this is all that we need in this paper.

\begin{definition}[\cite{heckel2023colouring}, Def.~2.3]\label{deftame} 
	Let $\mathbf{k}=\mathbf{k}(n)$ be a sequence of complete $(\alpha-1)$-bounded $k$-profiles. 
	Then $\bfk$ is called \emph{tame}  if both of the following conditions hold. 
	\begin{enumerate}[a)]
		\itemsep2pt 
\item		
		There is an increasing function $ \gamma: \mathbb{N}_0\rightarrow \mathbb{R}$ with  $\gamma(x)\rightarrow \infty$ as $x \rightarrow \infty$ so that, if $n$ is large enough,
		\[{u k_u}/{n} < 2^{-(\alpha-u)\gamma(\alpha-u)}\]
		 for all $1 \le u \le \alpha-1$.
		\label{tame:tail}
		\item \label{largeexp} There is a constant $ c\in(0,1)$ so that $ \ln \E_m \big[\bar{X}_\kp\big] \gg -n^{1-c}$.
	\end{enumerate}
\end{definition}
Note that a) implies that there is a function $u^*= u^*(n) \sim \alpha \sim 2 \log_2 n$ so that $k_u =0$ unless $u^* \le u \le \alpha-1$, and we will from now on use the function $u^*$ whenever we consider a tame profile $\bfk$ without further explanation. In particular, $k_1=0$. Furthermore, it follows that for any tame $k$-profile $\bfk$,
\[k \sim \frac{n}{2\log_2 n}.\]

\subsection{The random variables $Z_k$ and $Z^{\mathrm{co}}_k$}

Let $\bfk$ be a tame $(\alpha-1)$-bounded $k$-profile for some $k \ge 1$, and $\pi \in \Pi_\bfk$. In this subsection we define a suitable random variable $Z^\mathrm{co}_\bfk$, and later we will set $Z=Z^{\mathrm{co}}_{\bfk^*}$ for the special $k^*$-profile $\boldk^*$ as in Proposition~\ref{proposition:main}.

The definition imposes heavy structural conditions on pairs of partitions $\pi, \pi'$ affecting the second moment of $Z_\bfk$ while leaving the first moment almost unchanged from $\E_{1/2}[X_\bfk]$, as we will see in \eqref{eq:Zfirstmoment}. This is what made it feasible to obtain the bounds on $\E[Z_\bfk^2]$ in \cite{heckel2023colouring}.

 We start by recalling the following definition from \cite{heckel2023colouring} which quantifies how a set is split up by a partition.
\begin{definition}[\cite{heckel2023colouring}, Def.~4.1] \label{def:tcomposed}
	For an ordered partition $\pi=(V_1, \dots, V_k)$ of $V=[n]$ and a set $S \subset V$, we write
	\[
	z(S, \pi):= \big| \{i: V_i \cap S \neq \emptyset \big\}|.
	\]
	If $z(S, \pi)=z$, we say that $S$ is \emph{$z$-composed} with respect to $\pi$.  
\end{definition}

Additionally to the event $A_\pi$ that $\pi$ is a colouring, in \cite{heckel2023colouring} the following events were defined. (We tweak the definition of $D_\pi$ slightly, writing $ \frac 12 \ln^3 n$ instead of $\ln^3 n$.)

\begin{itemize}
 \item[$B_\pi:$] If $S \subset V$ is an independent set 
 with $u^* \le |S| \le \alpha-1$, then
 \[
 z(S,\pi)\le 2  ~~\text{ or }~~ z(S,\pi) \ge |S| - 2 (\alpha - |S|)-1.
 \]
 That is, $S$ is composed of vertices from at most two parts of $\pi$, or from at least $|S| - 2 (\alpha - |S|)-1$ parts.
 
 \item[$C_\pi:$] If $S\subset V$ is an independent set 
 such that $u^*\le |S|\le \alpha-1$ and such that $z(S,\pi)= 2$, then there are parts $V_i$ and $V_j$ such that 
 \[
 |S \cap V_{i}|= 1 ~~\text{ and }~~ |S \cap V_j|=|S|-1.
 \]
 
 \item[$D_\pi:$] There are at most $ \frac 12 \ln^3 n$ independent sets $S$ of size between $u^*$ and $\alpha-1$ such that 
 \[
 z(S,\pi)= 2
 \quad
 \text{and}
 \quad
 \text{there is a part $V_i$ such that  $|S \cap V_i| \ge |V_i|-1$}.
 \]
\end{itemize}
Now let
\[
Z_\bfk = \sum_{\pi \in \Pi_{\bfk}} \mathbbm{1}_{A_\pi \cap B_\pi \cap C_\pi \cap D_\pi}.
\]
In \cite{heckel2023colouring}, Lemmas 5.1, 5.2 and 5.3, it was shown that if $\kp$ is a tame $(\alpha-1)$-bounded profile, then (for both $G_{n,m}$ and $G_{n, 1/2}$) by the first moment method
\begin{equation} \label{eq:zeroevents}
\Pb(\overline{B_\pi} ~|~ A_\pi)=o(1), \quad  \Pb(\overline{C_\pi} ~|~ A_\pi)=o(1), \quad
\Pb(\overline{D_\pi} ~|~ A_\pi)= o(1)
\end{equation}
(it does not matter that we changed the definition of $D_\pi$ slightly, as the proof goes through exactly the same for $\frac 12 \log^3 n$ --- in the proof in \cite{heckel2023colouring} it is shown that the expected number of independent sets as in the definition of $D_{\pi}$ is $O(\log^2 n ) \ll \frac 12 \log^3 n$). This 
immediately implies
\begin{equation}\label{eq:Zfirstmomentold}
\E_m[Z_\bfk] \sim \E_m[X_\bfk] \text{ and } \E_{1/2}[Z_\bfk] \sim \E_{1/2}[X_\bfk]. 
\end{equation}
Furthermore, note that \eqref{eq:zeroevents} implies
\[
\Pb_{1/2}(\overline{B}_\pi ~|~ A^\mathrm{co}_\pi)=o(1), \quad  \Pb_{1/2}(\overline{C_\pi} ~|~ A^\mathrm{co}_\pi)=o(1), \quad
\Pb_{1/2}(\overline{D_\pi} ~|~ A^\mathrm{co}_\pi)= o(1).
\]
Indeed, note that $A^\mathrm{co}_{\pi}$ is the union of $2^k$ disjoint events $E_i, 1 \le i \le 2^k$, one for each choice of which colour class is a clique and which is an independent set. $G_{n, 1/2}$ conditional on some $E_i$ is simply $G_{n, 1/2}$ with some edges (the cliques) and non-edges (the independent sets) fixed. Since $\overline{B_\pi}$, $\overline{C_\pi}$, $\overline{D_\pi}$ are down-sets in the edges of $G_{n, 1/2}$, we have for example $\Pb_{1/2} (\overline{B_\pi}|E_i) \le \Pb_{1/2} (\overline{B_\pi}|A_\pi)$,
and so
\[
\Pb_{1/2} (\overline{B_\pi}~|~ A^\mathrm{co}_\pi) = \Pb_{1/2} (\overline{B_\pi}~|~ \bigcupdot_{i=1}^{2^k}E_i) \le \Pb_{1/2} (\overline{B_\pi}~|~A_\pi) = o(1).
\]
Define the events $B_\pi^\mathrm{clique}$, $C_\pi^\mathrm{clique}$, $D_\pi^\mathrm{clique}$ in the same way as $B_\pi$, $C_\pi$, $D_\pi$, only replacing `independent set' by `clique' in the definitions. By symmetry, it follows that 
\[
\Pb_{1/2}(\overline{B_\pi}^\mathrm{clique} ~|~ A^\mathrm{co}_\pi)=o(1), \quad  \Pb_{1/2}(\overline{C_\pi}^\mathrm{clique}  ~|~ A^\mathrm{co}_\pi)=o(1), \quad
\Pb_{1/2}(\overline{D_\pi}^\mathrm{clique}  ~|~ A^\mathrm{co}_\pi)= o(1).
\]
Set
\[
B_\pi^ \mathrm{co} = B_\pi \cap B_\pi^\mathrm{clique}, \quad C_\pi^ \mathrm{co} = C_\pi \cap C_\pi^\mathrm{clique}, \quad D_\pi^ \mathrm{co} = D_\pi \cap D_\pi^\mathrm{clique}
\]
and 
\begin{equation}\label{eq:defZ}
Z_\bfk^\mathrm{co} = \sum_{\pi \in \Pi_{\bfk}} \mathbbm{1}_{A_\pi^\mathrm{co} \cap B_\pi^\mathrm{co} \cap C_\pi^\mathrm{co} \cap D_\pi^\mathrm{co}},
\end{equation}
then the above implies immediately that
\begin{equation} \label{eq:Zfirstmoment}
\E_{1/2}[Z_\bfk^\mathrm{co}] \sim \E_{1/2}[X_\bfk^\mathrm{co}].
\end{equation}

\subsection{Relevant pairs}

Let $\kp$ be a tame $(\alpha-1)$-bounded $k$-profile. The point of the definition of $Z^\mathrm{co}_\bfk$ is that, when considering
\[
\E_{1/2}[(Z^\mathrm{co}_\bfk)^2] = \sum_{\pi, \pi' \in \Pi_\bfk} \Pb(A_\pi^\mathrm{co} \cap B_\pi^\mathrm{co} \cap C_\pi^\mathrm{co} \cap D_\pi^\mathrm{co} \cap A_{\pi'}^\mathrm{co} \cap B_{\pi'}^\mathrm{co} \cap C_{\pi'}^\mathrm{co} \cap D_{\pi'}^\mathrm{co}),
\]
the probability for a pair $(\pi, \pi')$ in the sum is $0$ unless the colour classes of $\pi'$ do not violate any of the events $B_\pi^\mathrm{co} \cap C_\pi^\mathrm{co} \cap D_\pi^\mathrm{co} $, and vice versa. The set of \emph{relevant pairs} --- those where the probability above is not $0$ --- is given in the following definition.

 \begin{definition}\label{def:tamepairs}[\cite{heckel2023colouring}, Def.~6.1]
 	Let $\pi=(V_i)_{i=1}^k$ and $\pi'=(V'_j)_{j=1}^k$ be ordered partitions with profile $\kp$. Then we say that $(\pi, \pi')$ is \emph{relevant} if all of the following hold.
 	\begin{enumerate}[a)]
 		\item 
 		\begin{enumerate}[1.]
 			\item If $V_j'$ is a part of $\pi'$ of size $u$, then 
 			\[z(V_j', \pi) \le 2\quad \text{or} \quad z(V_j', \pi) \ge u-2(\alpha-u)-1.\]
 			\item If $z(V_j', \pi) = 2$ for some part $V_j'$ of $\pi'$, then there are parts $V_{i_1}$ and $V_{i_2}$ of $\pi$ such that 
 			\[|V_j' \cap V_{i_1}|=1\quad \text{and} \quad |V_j'\cap V_{i_2}|=|V_j'|-1 .\]
 			\item There are at most $\ln^3 n$ parts $V_j'$ of $\pi'$ such that 
 			\[z(V_j', \pi)=2, \quad \text{and} \quad \text{there is a part $V_i$ of $\pi$ such that } |V_j' \cap V_i| \ge |V_i|-1.\] 
 		\end{enumerate}
 		\item All of the above also hold with $\pi$ and $\pi'$ swapped.
 	\end{enumerate}
 	Let $\Pi_\mathrm{rel}= \Pi_\mathrm{rel}(\kp)$ denote the set of all relevant pairs  $(\pi, \pi')$ of vertex partitions.
 \end{definition} 
Then we have\footnote{Note that the reason for slightly altering the definition of $D_\pi$ was so that we would still only have $\ln^3 n$ (and not $2\ln^3 n$) parts in part 3 of a) as we did in \cite{heckel2023colouring}, even though these parts now come from the two events $D_\pi$ and $D_\pi^\mathrm{clique}$.}
\begin{equation}\label{eq:secondmomentpirel}
\E_{1/2}[(Z^\mathrm{co}_\bfk)^2] = \sum_{(\pi, \pi') \in \Pi_\mathrm{rel}} \Pb(A_\pi^\mathrm{co} \cap B_\pi^\mathrm{co} \cap C_\pi^\mathrm{co} \cap D_\pi^\mathrm{co} \cap A_{\pi'}^\mathrm{co} \cap B_{\pi'}^\mathrm{co} \cap C_{\pi'}^\mathrm{co} \cap D_{\pi'}^\mathrm{co}) \le \sum_{(\pi, \pi') \in \Pi_\mathrm{rel}}  \Pb(A_\pi^\mathrm{co} \cap A_{\pi'}^\mathrm{co}),
\end{equation}
and our aim is to bound
\[
\frac{\E_{1/2}[(Z^\mathrm{co}_\bfk)^2]}{ \E_{1/2}[Z^\mathrm{co}_\bfk]^2} \le \frac 1 { \E_{1/2}[Z^\mathrm{co}_\bfk]^2}  \sum_{(\pi, \pi') \in \Pi_\mathrm{rel}}  \Pb(A_\pi^\mathrm{co} \cap A_{\pi'}^\mathrm{co})
< \exp(n^{0.99})\]
for an appropriate choice of profile $\bfk$.

\subsection{Bounding the second moment}

We  split up the second moment calculation into three ranges, depending on how similar the two partitions $(\pi, \pi')\in \Pi_\mathrm{rel}$ are. Given $\pi, \pi' \in \Pi_\bfk$ for some tame sequence $\kp=(k_u)_{u=u^*}^{\alpha-1}$ of profiles, let $\ell_u = \ell_u(\pi, \pi') \le k_u$ denote the number of parts of size $u$ which are the identical in $\pi$ and $\pi'$ (that is, the number of vertex sets $U$ such that $|U|=u$ and $U$ is a part both in $\pi$ and in $\pi'$), and let
\[
\lambda = \lambda(\pi, \pi') = \sum_{u^* \le u \le \alpha-1}  \ell_u u/n
\]
be the fraction of vertices contained in identical parts.  Following \cite{heckel2023colouring}, let $c \in (0,1)$ be the constant from the Definition~\ref{deftame} of the tame sequence $\bfk$, set $c_0=c/3 \in (0, \frac 13)$, and let
\begin{equation}
	\label{eq:defofscrambledetc}
	\begin{split}
		\Pi_{\text{scrambled}}
		&= \left\{ (\pi, \pi') \in \Pi_{\mathrm{rel}} \mid \lambda(\pi, \pi') < {\ln^{-3} n} \right\},  \\
		\Pi_{\text{middle}}
		&= \left\{ (\pi, \pi') \in \Pi_{\mathrm{rel}} \mid  {\ln^{-3} n} \le \lambda(\pi, \pi') \le 1-n^{-{c_0}}\right\},  \\
		\Pi_{\text{similar}}
		&= \left\{ (\pi, \pi') \in \Pi_{\mathrm{rel}} \mid \lambda(\pi, \pi') > 1- n^{-{c_0}}\right\}.
	\end{split}    
\end{equation}
Then we have the following results from \cite{heckel2023colouring}. They are phrased in a more general way in \cite{heckel2023colouring}, we state them here under the conditions we need them, that is, $p=1/2$, $a=\alpha-1$, 
 recalling that $k\sim \frac{n}{2\log_2n}$ if a $k$-profile $\kp$ is tame, and from \eqref{eq:mualpha} that $\mu_{\alpha-1} = \Theta(\mu_\alpha n/\log n)$ and $\mu_\alpha \ge 1$.

\begin{lemma}[\cite{heckel2023colouring}, Le.~6.3]
	\label{lemmascrambled} 
	Let $\kp=\kp(n)$ be a tame $(\alpha-1)$-bounded profile, then
	\begin{equation*}\frac{1}{\E_m[Z_\kp]^2}\sum_{(\pi,\pi') \in \Pi_{\textup{scrambled}}} \Pb_m\left(A_{\pi} \cap A_{\pi'}\right) \le \exp \big( 
		O( n / \mu_\alpha)+o(1)		\big).
		\end{equation*}
\end{lemma}

To state the next lemma, we need some notation: For two profiles $\bfk=(k_u)_u, \boldsymbol{\ell}=(\ell_u)_{u}$, we write $\boldsymbol{\ell} \le \bfk$ if $\ell_u \le k_u$ for all $u$. The condition \eqref{eq:lowerboundbeta} below is slightly technical; for the purposes of this paper it suffices to know that by Lemma~\ref{lemma:kstartame} later, the profile $\bfk^*$ which we will consider fulfils this condition.
\begin{lemma}[\cite{heckel2023colouring}, Le.~6.4] \label{lemmamiddle} 
Fix $\epsilon' > 0$ and suppose $\mu_\alpha \ge n^{\epsilon'}$, and let $\kp=\kp(n)$ be a tame $(\alpha-1)$-bounded profile. Furthermore, suppose that for any fixed ${\delta}>0$, if $n$ is large enough, then for all colouring profiles $\boldsymbol{\ell} \le \bfk$ such that ${\delta} \le \sum_{1 \le u \le \alpha-1} \lambda_u \le 1-{\delta}$ (where $\lambda_u = \ell_u u/n$),
\begin{equation}\label{eq:lowerboundbeta}
	\E_{1/2}\left[\bar{X}_{\boldsymbol{\ell}}\right] \ge \exp \Big( \ln^6 n \Big) \prod_{1 \le u \le a}{\binom{k_u}{\ell_u}}^2 .    \end{equation}
Then 
	\[
\frac{1}{\E_m[Z_\kp]^2}\sum_{(\pi,\pi') \in \Pi_{\textup{middle}}} \Pb_m\left(A_{\pi} \cap A_{\pi'}\right) =o(1).
	\]
\end{lemma}

\begin{lemma}[\cite{heckel2023colouring}, Le.~6.5]
	\label{lemmasimilar}
Let $\kp=\kp(n)$ be a tame $(\alpha-1)$-bounded profile, 
	then
	\[
\frac{1}{\E_m[Z_\kp]^2}\sum_{(\pi,\pi') \in \Pi_{\textup{similar}}} \Pb_m\left(A_{\pi} \cap A_{\pi'}\right) \le  \frac{n^{O(1)}}{\E_m[\bar X_\bfk]} \enspace .
	\]
\end{lemma} 

Now recall that, for a tame $(\alpha-1)$-bounded profile $\bfk$, all colour classes are in size between $u^* \sim 2 \log_2n$ and $\alpha-1 \sim 2 \log_2 n$. In particular, the profile $\bfk$ has
\[
f_\bfk = \sum_{u=u^*}^{\alpha-1} {u \choose 2} k_u \sim n \log_2 n
\]
`forbidden' edges (edges that may not be present for a partition $\pi$ to be a colouring). By this and Lemma~\ref{gnmlemma}, for $\pi, \pi \in \Pi_\bfk$,
\begin{equation}\label{eq:gnhtransfer}
\frac{\Pb_m(A_\pi)}{\Pb_{1/2}(A_\pi)}, ~~~~ \frac{\Pb_m(A_\pi \cap A_{\pi'})}{\Pb_{1/2}(A_\pi \cap A_{\pi'})}, ~~~~ \frac{\E_{m}[X_\bfk]}{ \E_{1/2}[X_\bfk]} =\exp(O(\log^2 n)).
\end{equation}
Using this, Proposition~\ref{prop:probabilities} and \eqref{eq:firstmomentcocol}, \eqref{eq:Zfirstmomentold},  \eqref{eq:Zfirstmoment}, we can directly transfer Lemmas~\ref{lemmascrambled} and~\ref{lemmamiddle} to our setting.
\begin{lemma}\label{lemmascrambledco}
Let $\kp=\kp(n)$ be a tame $(\alpha-1)$-bounded profile, then
	\begin{equation*}\frac{1}{\E_{1/2}[Z_\kp^\mathrm{co}]^2}\sum_{(\pi,\pi') \in \Pi_{\textup{scrambled}}} \Pb_{1/2}\left(A^\mathrm{co}_{\pi} \cap A^\mathrm{co}_{\pi'}\right) \le \exp \left(  	O(n / \mu_{\alpha}+\log^2 n)\right).
	\end{equation*}\qed
\end{lemma}
\begin{lemma}\label{lemmamiddleco}
	Under the conditions of Lemma~\ref{lemmamiddle},
	\[
	\frac{1}{\E_{1/2}[Z^\mathrm{co}_\kp]^2}\sum_{(\pi,\pi') \in \Pi_{\textup{middle}}} \Pb_{1/2}\left(A^\mathrm{co}_{\pi} \cap A^\mathrm{co}_{\pi'}\right) = \exp(O(\log^2 n)).
	\]\qed
\end{lemma}
With slightly more effort, we obtain the following result from Lemma~\ref{lemmasimilar}.
\begin{lemma}
	\label{lemmasimilarco}
Let $\kp=\kp(n)$ be a tame $(\alpha-1)$-bounded profile, 
	then
	\[
\frac{1}{\E_{1/2}[Z_\kp^\mathrm{co}]^2}\sum_{(\pi,\pi') \in \Pi_{\textup{similar}}} \Pb_{1/2}\left(A^{\mathrm{co}}_{\pi} \cap A^{\mathrm{co}}_{\pi'}\right) \le  \frac{\exp(O\big(n^{1-c_0}/ \log n)\big)}{\E_{1/2}[\bar X^{\mathrm{co}}_\bfk]} \enspace .
	\]
\end{lemma} 
\begin{proof}
	For $(\pi, \pi') \in \Pi_\textup{similar}$, as $\lambda=\lambda(\pi, \pi') >1-n^{-c_0}$, we have
	\[
	k-\ell = \sum_{u=u^*}^{\alpha-1}(k_u-\ell_u) \sim \frac{n}{2 \log_2 n}\sum_{u=u^*}^{\alpha-1}\frac{(k_u-\ell_u)u}{n} =   \frac{n}{2 \log_2 n}(1-\lambda) = O(n^{1-c_0} / \log n).
	\]
		By Proposition~\ref{prop:probabilities}, \eqref{eq:firstmomentcocol}, \eqref{eq:Zfirstmomentold}, \eqref{eq:Zfirstmoment} and \eqref{eq:gnhtransfer}, we have 
	\begin{align*}
	\frac{1}{\E_{1/2}[Z_\kp^\mathrm{co}]^2}&\sum_{(\pi,\pi') \in \Pi_{\textup{similar}}} \Pb_{1/2}\left(A^{\mathrm{co}}_{\pi} \cap A^{\mathrm{co}}_{\pi'}\right) \le \frac{1+o(1)}{2^k\E_{1/2}[Z_\kp]^2} \sum_{(\pi,\pi') \in \Pi_{\textup{similar}}} 2^{k-\ell}\Pb_{1/2}\left(A_{\pi} \cap A_{\pi'} \right) \\
	&= \frac{1}{2^k\E_m[Z_\kp]^2}\sum_{(\pi,\pi') \in \Pi_{\textup{similar}}} \Pb_m\left(A_{\pi} \cap A_{\pi'}\right) \exp \Big(O\Big(\frac{n^{1-c_0}}{\log n}\Big)\Big). 	\end{align*}
		Using Lemma~\ref{lemmasimilar} and \eqref{eq:firstmomentcocol}, \eqref{eq:gnhtransfer} again gives the result.
\end{proof}

\subsection{The special profile $\bfk^*$}

In this section we combine some results from \cite{heckel2023colouring} and \cite{HRHowdoes} that imply that a suitable $k^*$-profile exists that we may apply Lemmas \ref{lemmascrambledco}--\ref{lemmasimilarco} to. 

We start with some notation. The following quantity was first introduced in \cite{HRHowdoes} 
as an approximation of the logarithm of the expected number of $t$-bounded $k$-colourings. Let $P_{n,k,t}^0$ denote the set of $t$-bounded profiles $(k_u)_{u=1}^t$, and set
\begin{equation*}
	L_0(n,k,t):=\sup_{\pi \in P^0_{n,k,t}}\left\{ n\log n -n +k -\sum_{u=1}^t k_u \log(k_u d_u) \right\}
\end{equation*}
where
\begin{equation*}
	d_u:= 2^{\binom{u}{2}}u!.
\end{equation*}
For $t=O(\log n)$ and $1< n/k <t$, the logarithm $\log E_{n,k,t}$ of the expected number of unordered $t$-bounded $k$-colourings of $G_{n, 1/2}$ is within $O(\log^4 n)$ of $L_0(n,k,t)$ (see \cite{HRHowdoes}, Lemma~29). In particular, by the definition \eqref{ktdef} of $ \boldk_{\alpha-1}$, it follows that
\begin{equation} \label{eq:L0atthreshold}
	L_0 (n, \boldk_{\alpha-1}-1, \alpha-1) \le \Theta (\log^ 4 n) \quad \text{and} \quad 	L_0 (n, \boldk_{\alpha-1}, \alpha-1) \ge -\Theta (\log^ 4 n). 
\end{equation}

The following result will help us estimate how this expression changes as we decrease $k$.
\begin{lemma}[\cite{HRHowdoes}, Cor.~39]\label{lemma:delk}
	Uniformly over all $k\le n/2$ such that $k=n/(\alpha-1-\Theta(1))$,	
	\[
	\frac{\partial}{\partial k} L_0 (n, k, \alpha-1) = \frac 2 {\log 2} \log^2 n +O(\log n \log \log n).
	\]
\end{lemma}
The following is (a simplified version of) a lemma from \cite{heckel2023colouring}.

\begin{lemma}[\cite{heckel2023colouring}, Le.~7.20]\label{lemma:kstartame}
Fix $\epsilon>0$ and let $k^*= \boldk_{\alpha-1} - n^{1-\epsilon/2}$. Then there is an $(\alpha-1)$-bounded $k$-colouring profile $\bfk^*=\bfk^*(n)$ so that all of the following hold.
	\begin{itemize}
		\itemsep1pt
		\item[a)] 	$\ln \E_{1/2}[\bar X_{\bfk^*}] 
		= L_0(n,k^*,\alpha-1) +O(\log^{3/2} n)$.
		\item[b)] Part \ref{tame:tail}) of Definition \ref{deftame} of a tame profile holds for $\bfk^*$.
		\item[c)] If $\mu_{\alpha-1} \ge n^{1.05}$, then \eqref{eq:lowerboundbeta} holds.
			\end{itemize}
\end{lemma}
We also need the following lemma from \cite{heckel2023colouring} estimating $\boldk_{\alpha-1}$.
\begin{lemma}[\cite{heckel2023colouring}, Le.~7.4] \label{lemma:boldkestimate}
\[	\frac{n}{\boldk_{\alpha-1}} = \alpha_0-1-\frac 2 {\log 2} +o(1).
\]
\end{lemma}

\subsection{Completing the proof}\label{section:complete}
Now we have all ingredients in place. Let $\epsilon>0$ and set $k^*= \boldk_{\alpha-1} - n^{1-\epsilon/2} \sim \frac{n}{2 \log_2 n}$. By Lemma~\ref{lemma:boldkestimate} we have $k^* = \frac{n}{\alpha-1-\Theta(1)}$, so by \eqref{eq:L0atthreshold} and Lemma~\ref{lemma:delk},
\[L_0(n,k^*,\alpha-1) = -n^{1-\epsilon/2}\frac 2 {\log 2} \log^2 n +O(n^{1-\epsilon/2}\log n \log \log n). \]
Let $\bfk^*$ be the profile from Lemma~\ref{lemma:kstartame}, then in particular
\[
\ln \E_{1/2}[\bar X_{\bfk^*}] = -\Theta\big(n^{1-\epsilon/2} \log^2 n \big). 
\]
Together with \eqref{eq:gnhtransfer}, \eqref{eq:gnhtransfer} and part b) of Lemma~\ref{lemma:kstartame}, this implies that both conditions of Definition~\ref{deftame} are met with $c=\epsilon/4$, say. So the profile $\bfk^*$ is tame. We assume $\mu_\alpha \ge n^{0.05+\epsilon}$, and by \eqref{eq:mualpha} that implies $\mu_{\alpha-1} \ge n^{1.05}$ if $n$ is large enough. So by part c) of Lemma~\ref{lemma:kstartame}, we also have \eqref{eq:lowerboundbeta}.

The point is of course that the expected number of \emph{cocolourings} with profile $\bfk^*$ is much larger: as $k^* \sim \frac{n}{2 \log_2 n}$, by \eqref{eq:firstmomentcocol},
\[
\E_{1/2}[\bar X^\mathrm{co}_{\bfk^*}] = 2^k \E_{1/2}[\bar X_{\bfk^*}] = \exp\big(\Theta(n/\log n)\big).
\]
Now set $Z= Z_{\bfk^*}^\mathrm{co}$. By the above, and since we assumed $\mu_\alpha \ge n^{0.05+\epsilon}$, we may apply the three second moment lemmas, Lemmas~\ref{lemmascrambledco}--\ref{lemmasimilarco} (with $c_0=c/3 = \epsilon/12$) as well as \eqref{eq:secondmomentpirel} to bound
\[
\frac{\E_{1/2}[Z^2]}{\E_{1/2}[Z]^2} \le \frac{1}{\E_{1/2}[Z]^2} \sum_{(\pi, \pi') \in \Pi_{\textup{rel}}(\bfk^*)} \Pb_{1/2} (A^\mathrm{co}_{\pi} \cap A^\mathrm{co}_{\pi'}) \le \exp(O(n^{0.95})).
\]
This gives part b) of Proposition~\ref{proposition:main}. Of course by its definition \eqref{eq:defZ}, $Z=Z_{\bfk^*}^\mathrm{co}>0$ implies $X_{\bfk^*}^\mathrm{co}>0$, so we also have part a). This completes the proof of Proposition~\ref{proposition:main} and thereby Theorem~\ref{theorem:difference}.

\section{Discussion}\label{section:discussion}
The reason why we only prove Theorem~\ref{theorem:difference} for $\mu_\alpha \ge n^{0.05+\epsilon}$ is that 
$\mu_{\alpha-1} = \mu_\alpha n^{1-o(1)}\ge n^{1.05}$ is a condition in part c) of Lemma~\ref{lemma:kstartame} (which is Lemma~7.20 in \cite{heckel2023colouring}) which gives us the profile~$\bfk^*$. It should be straightforward to change the proof of Lemma~7.20 in   \cite{heckel2023colouring} to only require $\mu_\alpha \ge n^{x_0+\epsilon}$ for any $\epsilon>0$ and a certain constant $x_0 \approx 0.02905$. This would in turn yield Theorem~\ref{theorem:difference} for roughly 97\% of all values $n$. However, for $\mu_\alpha <n^{x_0}$, the behaviour of the optimal profile $\bfk^*$ changes --- see the discussion after Lemma~7.20 in \cite{heckel2023colouring}. Thus, while we still believe the conclusion of Theorem~\ref{theorem:difference} to be true in this case, we would need to analyse the optimal profile more carefully for such $n$ in order to obtain a similar result to Lemma~\ref{lemma:kstartame}. 

Could we improve the conclusion of Theorem~\ref{theorem:difference} from the lower bound $n^{1-\epsilon}$? We had quite a bit of room at the end of our proof in \S\ref{section:complete}, applying the second moment method to a $k^*$ where the expected number of $k^*$-cocolourings is still of order $\exp(\Theta(n/\log n))$. Indeed, the `first moment threshold' for the number of cocolourings should be  order $n / \log^3 n$ smaller than that of the normal chromatic number: the expected number of cocolourings is larger by a factor $2^k=\exp(\Theta(n/\log n))$ (see \eqref{eq:firstmomentcocol}), and decreasing the number of colours by $1$ multiplies the expected number of cocolourings by a factor $\exp(-\Theta(\log^2 n))$ (similarly to Lemma~\ref{lemma:delk}). We therefore have the following conjecture (which was first mentioned in \cite{heckel2024question}).
\begin{conjecture}
For $G \sim \Gnh$, whp, 
\[\chi(G)-\zeta(G)= \Theta(n/\log^3 n).\]
\end{conjecture}
The reason why we cannot prove the conjecture with the methods from this paper (not even for $n$ such that $n^{1.05+\epsilon} \le \mu_\alpha < n^{1-\epsilon}$) is that the Definition~\ref{deftame} of a tame profile $\bfk$ from \cite{heckel2023colouring} stipulates that $\E_m[X_\bfk] \ge \exp(-n^{1-c})$ for some constant $c>0$, and this would no longer be true for a $k^*$-profile if we let $k^*= \boldsymbol{k}_{\alpha-1}-\Theta(n/\log^3 n)$. A lot more work would be required to get around this issue --- either eliminating this constraint in \cite{heckel2023colouring}, or repeating parts of the analysis in \cite{heckel2023colouring} for cocolourings directly.


\begin{thebibliography}{14}
	\providecommand{\natexlab}[1]{#1}
	\providecommand{\url}[1]{\texttt{#1}}
	\expandafter\ifx\csname urlstyle\endcsname\relax
	\providecommand{\doi}[1]{doi: #1}\else
	\providecommand{\doi}{doi: \begingroup \urlstyle{rm}\Url}\fi
	
	\bibitem[erd()]{erdosproblem625}
	Erd{\H{o}}s {P}roblem \#625.
	\newblock \url{http://www.erdosproblems.com/625}.
	
	\bibitem[Bollob{\'a}s(1988)]{bollobas1988chromatic}
	B.~Bollob{\'a}s.
	\newblock The chromatic number of random graphs.
	\newblock \emph{Combinatorica}, 8\penalty0 (1):\penalty0 49--55, 1988.
	
	\bibitem[Bollob{\'a}s and Erd{\H{o}}s(1976)]{erdoscliques}
	B.~Bollob{\'a}s and P.~Erd{\H{o}}s.
	\newblock Cliques in random graphs.
	\newblock In \emph{Mathematical Proceedings of the Cambridge Philosophical
		Society}, volume~80, pages 419--427. Cambridge University Press, 1976.
	
	\bibitem[Erd{\H{o}}s and Gimbel(1993)]{erdos1993some}
	P.~Erd{\H{o}}s and J.~Gimbel.
	\newblock Some problems and results in cochromatic theory.
	\newblock In \emph{Annals of Discrete Mathematics}, volume~55, pages 261--264.
	Elsevier, 1993.
	
	\bibitem[Frieze(1990)]{frieze1990independence}
	A.~Frieze.
	\newblock On the independence number of random graphs.
	\newblock \emph{Discrete Mathematics}, 81\penalty0 (2):\penalty0 171--175,
	1990.
	
	\bibitem[Gimbel(2016)]{gimbel2016some}
	J.~Gimbel.
	\newblock Some of my favorite coloring problems for graphs and digraphs.
	\newblock \emph{Graph Theory: Favorite Conjectures and Open Problems-1}, pages
	95--108, 2016.
	
	\bibitem[Heckel(2021)]{heckel2019nonconcentration}
	A.~Heckel.
	\newblock Non-concentration of the chromatic number of a random graph.
	\newblock \emph{Journal of the American Mathematical Society}, 34\penalty0
	(1):\penalty0 245--260, 2021.
	
	\bibitem[Heckel(2024)]{heckel2024question}
	A.~Heckel.
	\newblock On a question of {E}rd{\H{o}}s and {G}imbel on the cochromatic
	number.
	\newblock \emph{The Electronic Journal of Combinatorics}, 31:\penalty0 P4.72,
	2024.
	
	\bibitem[Heckel and Panagiotou(2023)]{heckel2023colouring}
	A.~Heckel and K.~Panagiotou.
	\newblock Colouring random graphs: Tame colourings.
	\newblock \emph{arXiv preprint arXiv:2306.07253}, 2023.
	
	\bibitem[Heckel and Riordan(2023)]{HRHowdoes}
	A.~Heckel and O.~Riordan.
	\newblock How does the chromatic number of a random graph vary?
	\newblock \emph{Journal of the London Mathematical Society}, 108\penalty0
	(5):\penalty0 1769--1815, 2023.
	
	\bibitem[Janson et~al.(2000)Janson, {\L}uczak, and
	Ruci{\'n}ski]{janson:randomgraphs}
	S.~Janson, T.~{\L}uczak, and A.~Ruci{\'n}ski.
	\newblock \emph{Random Graphs}.
	\newblock Wiley, New York, 2000.
	
	\bibitem[Matula(1970)]{matula1970complete}
	D.~Matula.
	\newblock On the complete subgraphs of a random graph.
	\newblock In \emph{Proceedings of the 2nd Chapel Hill Conference on
		Combinatorial Mathematics and its Applications (Chapel Hill, NC, 1970)},
	pages 356--369, 1970.
	
	\bibitem[Matula(1972)]{matula1972employee}
	D.~Matula.
	\newblock The employee party problem.
	\newblock \emph{Notices of the American Mathematical Society}, 19\penalty0
	(2):\penalty0 A--382, 1972.
	
	\bibitem[Steiner(2024)]{steiner2024cochromatic}
	R.~Steiner.
	\newblock On the difference between the chromatic and cochromatic number.
	\newblock \emph{arXiv preprint arXiv:2408.02400v2, version 2}, 2024.
	
\end{thebibliography}

\end{document}